\documentclass[11pt]{amsart}
\usepackage[colorlinks, citecolor=blue, dvipdfm]{hyperref}

\usepackage{extarrows}
\usepackage{xypic}
\usepackage{exscale}
\usepackage{relsize}

\newtheorem{theorem}{Theorem}[section]
\newtheorem{lemma}[theorem]{Lemma}

\theoremstyle{definition}
\newtheorem{definition}[theorem]{Definition}

\newtheorem{example}[theorem]{Example}

\newtheorem{proposition}[theorem]{Proposition}
\newtheorem{corollary}[theorem]{Corollary}
\newtheorem{remark}[theorem]{Remark}
\newtheorem{conjecture}[theorem]{Conjecture}

\theoremstyle{remark}

\newcommand{\be}{\begin{equation}}
\newcommand{\ee}{\end{equation}}

\numberwithin{equation}{section}

\begin{document}
\title{K\"{a}hler hyperbolic manifolds and Chern number inequalities}
\author{Ping Li}
\address{School of Mathematical Sciences, Tongji University, Shanghai 200092, China}
\email{pingli@tongji.edu.cn\\
pinglimath@gmail.com}
\thanks{The author was partially supported by the National
Natural Science Foundation of China (Grant No. 11722109).}



\begin{abstract}
We show in this article that K\"{a}hler hyperbolic manifolds satisfy a family of optimal Chern number inequalities and the equality cases can be attained by some compact ball quotients.
These present restrictions to complex structures on negatively-curved compact K\"{a}hler manifolds, thus providing evidence to the rigidity conjecture of S.-T. Yau. The main ingredients in our proof are Gromov's results on the $L^2$-Hodge numbers, the $-1$-phenomenon of the $\chi_y$-genus and Hirzebruch's proportionality principle. Similar methods can be applied to obtain parallel results on K\"{a}hler non-elliptic manifolds. In addition to these, we term a condition called ``K\"{a}hler exactness", which includes K\"{a}hler hyperbolic and non-elliptic manifolds and has been used by B.-L. Chen and X. Yang in their work, and show that the canonical bundle of a K\"{a}hler exact manifold of general type is ample. Some of its consequences and remarks are discussed as well.
\end{abstract}

\maketitle

\section{Introduction}
Let us start the article by recalling two well-known conjectures related to the negativity of Riemannian sectional curvature, and their connections via the notion of ``K\"{a}hler hyperbolicity" introduced by Gromov (\cite{Gr}). The first one, usually attributed to Hopf, is
\begin{conjecture}[Hopf]\label{hopfconjecture}
The Euler characteristic $\chi(M)$ of a compact $2n$-dimensional Riemannian manifold $M$ with sectional curvature $K<0$ (resp. $K\leq 0$) satisfies $(-1)^n\chi(M)>0$ (resp. $(-1)^n\chi(M)\geq0$).
\end{conjecture}

This is true for $n=1$ and $2$ as the Gauss-Bonnet integrands in these two low-dimensional cases have the desired sign (\cite{Ch}) but is still open in its full generality for $n\geq 3$. Gromov introduced in \cite{Gr} the notion of ``K\"{a}hler hyperbolicity", which includes compact K\"{a}hler manifolds with negative (Riemannian) sectional curvature (``negatively-curved" for short) as special cases, and showed that the Euler characteristic of K\"{a}hler hyperbolic manifolds have the expected sign. As a consequence this settled Conjecture \ref{hopfconjecture} for K\"{a}hler manifolds when $K<0$. By extending Gromov's idea and notion above to nonnegative version, Cao-Xavier and Jost-Zuo (\cite{CX}, \cite{JZ}) independently introduced the concept of ``K\"{a}hler non-ellipticity" and established a parallel result and consequently settled Conjecture \ref{hopfconjecture} in the case of $K\leq 0$ for K\"{a}hler manifolds.

The second conjecture, which is due to S.-T. Yau (\cite[p. 678]{Ya2}) and can be viewed as a generalization of the classical Mostow rigidity theorem, is
\begin{conjecture}[S.-T. Yau]\label{yauconjecture}
The complex structure of a negatively-curved compact K\"{a}hler manifold $M$ with $\text{dim}_{\mathbb{C}}M\geq2$ is unique.
\end{conjecture}
This was solved by F. Zheng (\cite{Zh}) when $\text{dim}_{\mathbb{C}}M=2$. By introducing in \cite{Si} the notion of ``strongly negative curvature", which is slightly stronger than the negativity of sectional curvature, Y.-T. Siu showed that a compact K\"{a}hler manifold homotopy equivalent to a compact K\"{a}hler manifold with strongly negative curvature is either holomorphic or anti-holomorphic to it, thus establishing the most general form of Conjecture \ref{yauconjecture} to date.

With these materials in mind, a natural question related to negatively-curved compact K\"{a}hler manifolds arises: whether the extra condition of K\"{a}hlerness can lead to more constraints on their geometry and/or topology rather than merely saying that their Euler characteristic has the desired sign? On the other hand, if we are really able to deduce various geometric restrictions on them, these would provide some positive evidence towards Conjecture \ref{yauconjecture}.

Recently B.-L. Chen and X. Yang made some important progress towards this question and the Hopf Conjecture \ref{hopfconjecture} in two articles \cite{CY} and \cite{CY2}. In the first one \cite{CY}, They showed that a compact K\"{a}hler manifold homotopy equivalent to a negatively-curved compact Riemannian manifold admits a K\"{a}hler-Einstein metric of negative Ricci curvature (\cite[Thm 1.1]{CY}). In fact they deduced this from the Aubin-Yau theorem by noting that the canonical bundle of a K\"{a}hler hyperbolic manifold is ample (\cite[Thm 2.11]{CY}). Thanks to Yau's Chern number inequality (\cite{Ya1}),
this implies that a complex $n$-dimensional K\"{a}hler hyperbolic manifold $M$ satisfies \be\label{yauchernnumber}c_2(-c_1)^{n-2}[M]\geq \frac{n}{2(n+1)}(-c_1)^n[M],\ee
with equality holds if and only if $M$ is covered by the unit ball in $\mathbb{C}^n$. In their second article \cite{CY2}, they presented some sufficient conditions related to K\"{a}hler forms and fundamental groups for compact K\"{a}hler manifold to be K\"{a}hler hyperbolic or non-elliptic (\cite[Thms 1.5, 1.6, 1.7]{CY2}). Consequently this settles the Hopf Conjecture \ref{hopfconjecture} in these situations. One of their sufficient conditions involved shall be termed ``K\"{a}hler exactness" in our article (see Definition \ref{def}).

\emph{The main purpose} of this article is to take a step further towards this question by showing that K\"{a}hler hyperbolic manifolds as well as K\"{a}hler non-elliptic manifolds indeed satisfy a family of optimal hern number inequalities (Theorems \ref{theorem1} and \ref{theorem2}). In addition to these, we shall term a condition called ``K\"{a}hler exactness" used in \cite{CY2}, which include K\"{a}hler hyperbolic and non-elliptic manifolds, and show that a K\"{a}hler exact manifold of general type has ample canonical bundle (Theorem \ref{theorem3}).

\section*{Outline of this article}
The rest of this article is structured as follows. In Section \ref{section2} our main results in this article (Theorems \ref{theorem1}, \ref{theorem2} and \ref{theorem3}) as well as their corollaries are stated, and along this line we set up some necessary notation and terminology. Sections \ref{section3} and \ref{section4} are devoted to some background materials related to the proofs of main results. To be more precise, we
review in Section \ref{section3} the Hirzebruch $\chi_y$-genus, its $-1$-phenomenon and Hirzebruch's proportionality principle, which are the starting points of Theorem \ref{theorem1}. Then in Section \ref{section4} we briefly recall the concept of $L^2$-Hodge numbers, the relationship with the usual Hodge numbers via Atiyah's $L^2$-index theorem, and some vanishing-type results on K\"{a}hler hyperbolic and non-elliptic manifolds. With these preliminaries in hand, in the last Section \ref{section5} we shall give the desired proofs of main results.

\section*{Acknowledgements}
Part of this article was completed when the author visited the Fields Institute in Toronto in May 2018. The author would like to thank it for the hospitality.

\section{Main results}\label{section2}
Before stating the main results, let us recall several notions due to Gromov (\cite{Gr}) and Hirzebruch (\cite{Hi2}) respectively.

Assume that $(M,g)$ is a Riemannian manifold and $\pi:$ $(\widetilde{M},\widetilde{g})\rightarrow(M,g)$
 the universal covering with $\widetilde{g}:=\pi^{\ast}(g)$. A (necessarily exact) differential form $\alpha$ on $(M,g)$ is called \emph{$d$-bounded} if $\alpha=d\beta$ and the norm $$\big|\beta\big|_g:=\sup_{x\in M}\big|\beta(x)\big|_{g(x)}<\infty.$$
 A form $\alpha$ on $(M,g)$ is called \emph{$\widetilde{d}$-bounded} if $\pi^{\ast}(\alpha)$ is $d$-bounded on $(\widetilde{M},\widetilde{g})$. This concept is interesting only if $\widetilde{M}$ is non-compact. With this understood, a compact K\"{a}hler manifold is called \emph{K\"{a}hler hyperbolic} if it admits a K\"{a}hler metric such that its associated K\"{a}hler form is $\widetilde{d}$-bounded (\cite[p. 265]{Gr}). Obviously this definition is meaningful for only non-compact $\widetilde{M}$.

 Whether or not a form $\alpha$ is $\widetilde{d}$-bounded has homotopy invariance and depends only on its cohomology class $[\alpha]$, provided that the manifold $M$ in question is compact, and all bounded closed $k$-forms ($k\geq 2$) on a complete Riemannian manifold with sectional curvature bounded above by a negative constant are $\widetilde{d}$-bounded, which were all observed by Gromov in \cite{Gr} and detailed proofs can be founded in \cite{CY}. Typical examples of K\"{a}hler hyperbolic manifolds include compact K\"{a}hler manifolds homotopy equivalent to negatively-curved compact Riemannian manifolds, compact quotients of the bounded homogeneous symmetric domains in $\mathbb{C}^n$, and their submanifolds and products (\cite[p. 265]{Gr}).

Given a compact complex $n$-dimensional manifold $M$, one can associate to a polynomial $\chi_y(M)\in\mathbb{Z}[y]$, called the \emph{Hirzebruch $\chi_y$-genus}, in terms of their Hodge numbers $h^{p,q}(M)$ as follows.
\be\label{chiy}
\begin{split}
\chi_y(M):=\sum_{p=0}^n\chi^p(M)\cdot y^p
:=\sum_{p=0}^n\big[\sum_{q=0}^n(-1)^qh^{p,q}(M)\big]y^p.
\end{split}\ee
For instance,
$$\chi_y(\mathbb{C}P^n)=\sum_{p=0}^n(-y)^p.$$ It is known that these
$\chi^p(M)$ $(0\leq p\leq n)$ are indices of some Dolbeault-type elliptic operators and the Hirzebruch-Riemann-Roch theorem tells us that $\chi^p(M)$ can be expressed in terms of rationally linear combinations of Chern numbers, and $\chi^0(M)$ is nothing but the Todd genus of $M$. For more details on this subject we refer the reader to Section \ref{section3}.

With these concepts understood, now comes our first main result in this article.
\begin{theorem}\label{theorem1}
Suppose that $M$ is a complex $n$-dimensional K\"{a}hler hyperbolic manifold. Then $M$ satisfies $[\frac{n}{2}]+1$ optimal Chern number inequalities
\be\label{chernnumberinequality}
\begin{split}
A_i(c_1,\ldots,c_n)[M]&\geq(-1)^nA_i\Big({n+1\choose 1},\ldots,{n+1\choose n}\Big)\\
&=(-1)^nA_{i}
(c_1,\ldots,c_n)[\mathbb{C}P^n],\qquad0\leq i\leq[\frac{n}{2}],
\end{split}\ee
which can be determined by a recursive algorithm, and whose first three terms read as follows
\begin{eqnarray}\label{kahlerricciformtensor}
\left\{ \begin{array}{ll}
A_0(c_1,\ldots,c_n)[M]=(-1)^nc_n[M]\geq n+1,\\
~\\A_1(c_1,\ldots,c_n)[M]=(-1)^n\Big[\frac{n(3n-5)}{2}c_{n}+c_{1}c_{n-1}\Big][M]
\geq2(n-1)n(n+1),\\
~\\
\begin{split}
A_2(c_1,\ldots,c_n)[M]=&(-1)^n\Big[n(15n^{3}-150n^{2}+485n-502)c_{n}+
4(15n^{2}-85n+108)c_{1}c_{n-1}\\
&+8(c_{1}^{2}+3c_{2})c_{n-2}-
8(c_{1}^{3}-3c_{1}c_{2}+3c_{3})c_{n-3}\Big][M]\\
\geq&(-1)^nA_2\Big({n+1\choose 1},\ldots,{n+1\choose n}\Big).\end{split}
\end{array} \right.\nonumber
\end{eqnarray}
Furthermore,
\begin{enumerate}
\item
all the equality cases in (\ref{chernnumberinequality}) hold if $M$ is a compact ball quotient with $\chi^0=(-1)^n$,

\item
the $i$-th equality case in (\ref{chernnumberinequality}) holds if and only if
\be\label{equalitycase1}\chi^p(M)=(-1)^{n-p},\qquad 2i\leq p\leq n,\ee

and

\item
any equality case in the first $[\frac{n+1}{4}]+1$ ones in (\ref{chernnumberinequality}) holds if and only if \be\label{equalitycase2}\chi_y(M)=(-1)^n\chi_y(\mathbb{C}P^n).\ee
\end{enumerate}
\end{theorem}

\begin{remark}~
\begin{enumerate}
\item
The first inequality in (\ref{chernnumberinequality}), $$(-1)^nc_n[M]\geq n+1,$$
 is exactly an improved form of the inequality expected by the Hopf conjecture.
\item
It is interesting to see that both the equality case in (\ref{yauchernnumber}) and those in (\ref{chernnumberinequality}) are achieved by some compact quotients of the unit ball in $\mathbb{C}^n$. Nevertheless, in contrast to (\ref{yauchernnumber}), we do not know if they are also necessary to the equality cases in (\ref{chernnumberinequality}).
\item
As $i$ increases the formula $A_i(c_1,\ldots,c_n)$ involves progressively more and more Chern numbers, which shall be clear in Section \ref{section3}.
\end{enumerate}
\end{remark}

Compact K\"{a}hler manifolds homotopy equivalent to negatively-curved compact Riemannian manifolds are K\"{a}hler hyperbolic, as previously mentioned. So Theorem \ref{chernnumberinequality} yields the following consequence, which gives constraints on possible complex structures on such manifolds and thus provides some positive evidence to Yau's Conjecture \ref{yauconjecture}.
\begin{corollary}\label{coro1}
Compact K\"{a}hler manifolds homotopy equivalent to negatively-curved compact Riemannian manifolds satisfy the Chern number inequalities in (\ref{chernnumberinequality}) and various characterizations of their equality cases. In particular, they satisfy
\begin{eqnarray}\label{kahlerricciformtensor}
\left\{ \begin{array}{ll}
(-1)^nc_n[M]\geq n+1,\\
~\\(-1)^n\Big[\frac{n(3n-5)}{2}c_{n}+c_{1}c_{n-1}\Big][M]
\geq2(n-1)n(n+1).
\end{array} \right.\nonumber
\end{eqnarray}
When $n\geq 2$ (resp. $n\geq 3$), the first (resp. second) equality holds if and only if $\chi_y(M)=(-1)^n\chi_y(\mathbb{C}P^n)$.
\end{corollary}

In order to attack Conjecture \ref{hopfconjecture} in the K\"{a}hlerian case when $K\leq 0$ by extending Gromov's idea, Cao-Xavier and Jost-Zuo (\cite{CX}, \cite{JZ}) independently introduced the concept of ``K\"{a}hler non-ellipticity", which includes nonpositively curved compact K\"{a}hler manifolds,  and showed that their Euler characteristics have the desired property. A (necessarily exact) differential form $\alpha$ on a complete Riemannian manifold $(M,g)$ is called \emph{$d$-sublinear} if $\alpha=d\beta$ and
$$\big|\beta(x)\big|_{g(x)}\leq c\big[1+\rho(x,x_0)\big],\qquad\forall ~x\in M,$$
where $c$ is a constant and $\rho(x,x_0)$ stands for the Riemannian distance between $x$ and a base point $x_0$. Clearly a $d$-bounded form is $d$-sublinear. This $\alpha$ is called \emph{$\widetilde{d}$-sublinear} if $\pi^{\ast}(\alpha)$ is $d$-sublinear on the universal covering $(\widetilde{M},\widetilde{g})$. A compact K\"{a}hler manifold is called \emph{K\"{a}hler non-elliptic} if it admits a K\"{a}hler metric such that its associated K\"{a}hler form is $\widetilde{d}$-sublinear. Similar to K\"{a}hler hyperbolic manifolds, it also turns out that any bounded and closed form on a complete nonpositively-curved Riemannian manifold is $\widetilde{d}$-sublinear and the property of  $\widetilde{d}$-sublinearity has homotopy invariance (\cite{CX}).

With these understood, we have the following result for K\"{a}hler non-elliptic manifolds by applying a similar idea to the proof in Theorem \ref{theorem1}.
\begin{theorem}\label{theorem2}
Any K\"{a}hler non-elliptic manifold satisfy the following $[\frac{n}{2}]+1$ sharp Chern number inequalities:
\be\label{nonelliptic}(-1)^nA_i(c_1,\ldots,c_n)[M]\geq0,\qquad 0\leq i\leq[\frac{n}{2}].\ee
In particular, these inequalities hold for compact K\"{a}hler manifolds homotopy equivalent to nonpositively-curved compact Riemannian manifolds.
\end{theorem}
\begin{remark}
The sharpness of (\ref{nonelliptic}) can be easily seen from the examples of complex tori as they are K\"{a}hler non-elliptic and their Chern numbers vanish.
\end{remark}
In addition to the main results in \cite{Gr}, Gromov showed that a K\"{a}hler hyperbolic manifold is of general type, and asked if its canonical bundle is ample (\cite[p. 267]{Gr}). This was affirmatively answered by Chen and Yang in \cite[Thm 2.11]{CY} based on some observations in algebraic geometry and they applied it to deduce one of their main results (\cite[Thm 1.1]{CY}).

\emph{Our second main purpose} in this article is to generalize the concepts of K\"{a}hler hyperbolicity and non-ellipticity by terming a condition by ``K\"{a}hler exactness", which has been used in \cite{CY2}, and show that a K\"{a}hler exact manifold of general type has ample canonical bundle. Recall that on a \emph{compact} K\"{a}hler manifold any K\"{a}hler form is closed but can \emph{never} be exact, which motivates us to introduce the following notion.
\begin{definition}\label{def}
Let $\omega$ be a K\"{a}hler form on a compact K\"{a}hler manifold $M$ and $\pi:~\widetilde{M}\longrightarrow M$ the universal covering. This $\omega$ is called a \emph{K\"{a}hler exact} form if $\pi^{\ast}\omega$ is an exact $2$-form on $\widetilde{M}$, i.e., there exists a (globally defined) $1$-form $\beta$ on $\widetilde{M}$ such that $\pi^{\ast}\omega=d\beta$. A compact K\"{a}hler manifold is called \emph{K\"{a}hler exact} if it admits a K\"{a}hler exact form.
\end{definition}
\begin{remark}~
\begin{enumerate}
\item
$M$ is compact exact only if its universal covering $\widetilde{M}$ is non-compact.
\item
By definitions K\"{a}hler hyperbolic and non-elliptic manifolds, and particularly compact K\"{a}hler manifolds homotopy equivalent to nonpositively curved compact Riemannian manifolds are K\"{a}hler exact.
Chen-Yang gave in \cite{CY2} some sufficient conditions for K\"{a}hler exact manifolds to be K\"{a}hler hyperbolic or non-elliptic.
\item
It is immediate from the definition that compact complex submanifolds of K\"{a}hler exact manifolds are still K\"{a}hler exact.
\end{enumerate}
\end{remark}

Inspired by \cite[Thm 2.11]{CY}, we shall show in Section \ref{section5} the following result.
\begin{theorem}\label{theorem3}
Suppose that $M$ is a K\"{a}hler exact manifold of general type. Then the canonical bundle of $M$ is ample. This implies that $M$ admits a K\"{a}hler-Einstein metric of negative Ricci curvature and satisfies the Chern number inequality (\ref{yauchernnumber}).
\end{theorem}
An immediate corollary of Theorem \ref{theorem3} is the following result, which is the counterpart to \cite[Thm 1.1]{CY}.
\begin{corollary}\label{coro3}
If a compact K\"{a}hler manifold of general type is homotopy equivalent to a nonpositively curved compact Riemannian manifold, then its canonical bundle is ample and thus it admits a K\"{a}hler-Einstein metric of negative Ricci curvature and satisfies the Chern number inequality (\ref{yauchernnumber}).
\end{corollary}
Theorem \ref{theorem3} and Corollary \ref{coro3} are closely related to two conjectures of S. Kobayashi and F. Zheng respectively. Recall that a compact complex manifold $M$ is called \emph{Kobayashi hyperbolic} if every holomorphic map $f:~\mathbb{C}\rightarrow M$ is constant. The following two conjectures related to Kobayashi hyperbolicity are due to S. Kobayashi (\cite[p. 370]{Ko}) and F. Zheng (\cite[Thm 2]{Zh2}) respectively.
\begin{conjecture}[Kobayashi]\label{kobayashiconjecture}
If a compact K\"{a}hler manifold is Kobayashi hyperbolic, then its canonical bundle must be ample.
\end{conjecture}
\begin{conjecture}[Zheng]\label{zhengconjecture}
If a nonpositively curved compact K\"{a}hler manifold is of general type, it must be Kobayashi hyperbolic.
\end{conjecture}
Conjecture \ref{zhengconjecture} was verified by Zheng himself in dimension two (\cite[Thm 2]{Zh2}).
Gromov pointed out in \cite[p. 266]{Gr} that
K\"{a}hler hyperbolicity implies Kobayashi hyperbolicity. We refer the reader to \cite[Thm 1.2]{CX} for an extension and a detailed proof. If Conjecture \ref{kobayashiconjecture} were true, then \cite[Thm 2.11]{CY} would follow immediately. In view of the fact that K\"{a}hler exact manifolds to some extent are generalizations of K\"{a}hler hyperbolic and non-elliptic manifolds, Theorem \ref{theorem3} presents some positive evidence to Conjecture \ref{kobayashiconjecture}. If both Conjectures \ref{kobayashiconjecture} and \ref{zhengconjecture} were true, then a nonpositively curved compact K\"{a}hler manifold of general type would has ample canonical bundle, which is a special case of Corollary
\ref{coro3} and has been observed in \cite[\S 2.4]{Zh2}. So Corollary \ref{coro3} presents some positive evidence to Conjectures \ref{kobayashiconjecture} and \ref{zhengconjecture} somehow.

\section{Hirzebruch's $\chi_y$-genus and proportionality principle}\label{section3}
We briefly review the notion of the $\chi_y$-genus, its $-1$-phenomenon and Hirzebruch's proportionality principle respectively in the following three subsections.
\subsection{The Hirzebruch $\chi_y$-genus}
The $\chi_y$-genus was first introduced by Hirzebruch in his seminal book
\cite{Hi2} for projective manifolds and can be calculated via his celebrated Hirzebruch-Riemann-Roch theorem. The later Atiyah-Singer index theorem
implies that it still holds for general compact (almost-)complex manifolds. To be
more precise, let $(M,J)$ be a compact complex manifold
with $\text{dim}_{\mathbb{C}}M=n$ and complex structure $J$. As
usual we denote by $\bar{\partial}$ the $d$-bar operator which
acts on the complex vector spaces $\Omega^{p,q}(M)$ ($0\leq p,q\leq
n$) of $(p,q)$-type complex-valued differential forms on $(M,J)$. The choice of a Hermitian metric
on $(M,J)$ enables us to define the formal adjoint
$\bar{\partial}^{\ast}$ of the
$\bar{\partial}$-operator. Then for each $0\leq p\leq n$, we have
the following Dolbeault-type elliptic operator $D_p$:
\be\label{GDC}D_p:=\bar{\partial}+\bar{\partial}^{\ast}:~\bigoplus_{\textrm{$q$
even}}\Omega^{p,q}(M)\longrightarrow\bigoplus_{\textrm{$q$
odd}}\Omega^{p,q}(M),\ee whose index is denoted by $\chi^{p}(M)$ in
the notation of Hirzebruch in \cite{Hi2}. The Hirzebruch
$\chi_{y}$-genus, denoted by $\chi_{y}(M)$, is the generating
function of these indices $\chi^p(M)$:
$$\chi_{y}(M):=\sum_{p=0}^{n}\chi^{p}(M)\cdot y^{p}.$$

By definition
\be\begin{split}
\chi^p(M)=&\text{ind}(D_p)\\
=&\text{dim}_{\mathbb{C}}(\text{ker}D_p)-
\text{dim}_{\mathbb{C}}(\text{coker}D_p)\\
=&\text{dim}_{\mathbb{C}}\bigoplus_{\text{$q$ even}}\mathcal{H}^{p,q}_{\bar{\partial}}(M)-
\text{dim}_{\mathbb{C}}\bigoplus_{\text{$q$ odd}}\mathcal{H}^{p,q}_{\bar{\partial}}(M)\\
=&\sum_{q=0}^n(-1)^qh^{p,q}(M),
\end{split}\ee
where $\mathcal{H}^{p,q}_{\bar{\partial}}(M)$ are the spaces of complex-valued $\bar{\partial}$-harmonic forms and $h^{p,q}(M)$ the Hodge numbers of $M$. Consequently $\chi_y(M)$ has the desired expression (\ref{chiy}):
$$\chi_y(M)=\sum_{p=0}^n\big[\sum_{q=0}^n(-1)^qh^{p,q}\big]y^p.$$

The general form of the Hirzebruch-Riemann-Roch theorem, which is a
corollary of the Atiyah-Singer index theorem, allows us to compute
$\chi_y(M)$ in terms of the Chern numbers of $M$ as follows
\be\label{HRR}\chi_y(M)=\mathlarger{\int}_M\prod_{i=1}^n\frac{x_i(1+ye^{-x_i})}{1-e^{-x_i}},\ee
where $x_1,\ldots,x_n$ are formal Chern roots of $(M,J)$, i.e., the
$i$-th elementary symmetric polynomial of $x_1,\ldots,x_n$ represents the
$i$-th Chern class of $(M,J)$:
 $$c_1=x_1+\cdots+x_n,\qquad c_2=\sum_{1\leq i<j\leq n}x_ix_j,\qquad\ldots,\qquad c_n=x_1x_2\cdots x_n.$$
 This $\chi_y(M)$ famously satisfies
$$\chi_y(M)=(-y)^n\cdot\chi_{y^{-1}}(M),$$
which are equivalent to the
relations $\chi^p=(-1)^n\chi^{n-p}$ and can be derived from either
(\ref{HRR}) or the Serre duality for the Hodge numbers
(\cite[p. 102]{GH}):
\be\label{chiprelation}
\begin{split}
\chi^{p}=\sum_{q=0}^{n}(-1)^{q}h^{p,q}=&
\sum_{q=0}^{n}(-1)^{q}h^{n-p,n-q}\\
=&(-1)^n\sum_{q=0}^{n}(-1)^{q}h^{n-p,q}\\
=&(-1)^n\chi^{n-p}.
\end{split}\ee
For three values of $y$, this $\chi_y$-genus is an
important invariant: $\chi_y(M)\big|_{y=-1}$ is the Euler
characteristic of $M$, $\chi_y(M)\big|_{y=0}=\chi^0(M)$ is the Todd genus of
$M$, and $\chi_y(M)\big|_{y=1}$ is the signature of $M$.

\subsection{The $-1$-phenomenon}
The purpose of this subsection is to recall a $-1$-phenomenon
for the $\chi_y$-genus.

Note that when $n$ are small, the formulas of $\chi^p$ in terms of rationally linear combinations of Chern numbers can be explicitly written down. For example, $\chi^0$ were listed in \cite[p. 14]{Hi2} when $n\leq 6$. However, these formulas become more and more complicated as $n$ increases. So for \emph{general} $n$ there are \emph{no explicit} formulas for these $\chi^p$. Nevertheless, as we have mentioned, when evaluated at $y=-1$, $\chi_y(M)\big|_{y=-1}$
gives the Euler characteristic, which is equal to the top Chern
number $c_n[M]$. Note that $\chi_y(M)\big|_{y=-1}$ is exactly
the constant term in the Taylor expansion of $\chi_y(M)$ at $y=-1$.
Indeed, several independent articles (\cite{NR}, \cite{LW},
\cite{Sa1}), with different backgrounds, observed that, when
expanding the right-hand side of (\ref{HRR}) at $y=-1$, its first few coefficients for \emph{general $n$} have \emph{explicit} formulas in
terms of Chern numbers. More precisely, we have the following proposition.
\begin{proposition}\label{ch}
If we denote by $K_j(M)$ $(0\leq j\leq n)$ the coefficients in the Taylor expansion of $\chi_y(M)$ at $y=-1$, i.e.,
\be\label{chiy-1}\mathlarger{\int}_M\prod_{i=1}^n\frac{x_i(1+ye^{-x_i})}{1-e^{-x_i}}
=:\sum_{j=0}^nK_j(M)\cdot(y+1)^j,\ee
then we have
\begin{enumerate}
\item
any $K_{2i+1}$ is a linear combination of $K_{2j}$ for $0\leq j\leq i$ and so \emph{we are only interested in $K_{2i}$ for $0\leq i\leq[\frac{n}{2}]$},

\item
only the Chern classes
$$c_1,c_2,\dots,c_{2i-1},c_{n-2i+1},c_{n-2i+2},\ldots,c_n$$ are involved in the formula $K_{2i}$,

\item
there is a recursive algorithm to determine the formulas $K_{2i}$,\\

and\\

\item
the first few terms are given by
\begin{eqnarray}\label{firstfewterms}
\left\{ \begin{array}{ll}
K_0(M)=c_{n}[M],\\
~\\
K_1(M)= -\frac{1}{2}nc_{n}[M],\\
~\\
K_2(M)=
\frac{1}{12}\Big[\frac{n(3n-5)}{2}c_{n}+c_{1}c_{n-1}\Big][M]\\
~\\
K_3(M)=-\frac{1}{24}\Big[\frac{n(n-2)(n-3)}{2}c_{n}+
(n-2)c_{1}c_{n-1}\Big][M]\\
~\\
\begin{split}
K_4(M)=\frac{1}{5760}\Big[&n(15n^{3}-150n^{2}+485n-502)c_{n}+
4(15n^{2}-85n+108)c_{1}c_{n-1}\\
&+8(c_{1}^{2}+3c_{2})c_{n-2}-
8(c_{1}^{3}-3c_{1}c_{2}+3c_{3})c_{n-3}\Big][M].
\end{split}
\end{array} \right.\nonumber
\end{eqnarray}
\end{enumerate}
\end{proposition}
\begin{proof}
$(1)$ can be seen in \cite[Lemma 2.1]{Li3}. $(2)$ is presented in \cite[p. 300]{Sa2}. A recursive algorithm for calculating $K_{j}$ was described in \cite[p. 144]{LW}.
The formulas $K_j$ for $j\leq6$ are presented respectively in \cite[p. 141-143]{LW}, \cite[p. 145]{Sa1} and \cite[p. 300]{Sa2}.
\end{proof}
For the reader's convenience, we would like to end this subsection by briefly describing the history of the discoveries for these formulas and their applications, due to the author's best knowledge.

The formula $K_2$ appears implicitly in \cite[p. 18]{NR} and explicitly in \cite[p. 141-143]{LW}. Narasimhan-Ramanan applied $K_2$ to give a topological restriction on some moduli
spaces of stable vector bundles over Riemann surfaces. Libgober-Wood applied $K_2$ to prove the
uniqueness of the complex structure on K\"{a}hler manifolds of
certain homotopy types \cite[Thms 1, 2]{LW}. Salamon applied $K_2$ to
obtain a restriction on the Betti numbers of hyperK\"{a}hler
manifolds (\cite[Coro. 3.4, Thm 4.1]{Sa1}). In \cite{Hi3}, Hirzebruch
applied $K_1$, $K_2$ and $K_3$ to deduce a divisibility result on
the Euler number of almost-complex manifolds with $c_1=0$.
Inspired by these, the author investigated in \cite{Li2} and \cite{Li3} similar phenomena in pluri-$\chi_y$-genus and elliptic genus and uniformly termed them by
\emph{``$-1$-phenomena"}. In a recent article \cite{De2}, Debarre extended the aforementioned Libgober-Wood's ideas to refine their results as well as presented the formulas $K_j$ when $n\leq 9$.

\subsection{Hirzebruch's proportionality principle}
Let $X$ be a bounded homogeneous symmetric domain in $\mathbb{C}^n$, which is a non-compact Hermitian symmetric space. Dual to $X$ there is a naturally associated compact type Hermitian symmetric space $\widetilde{X}$. Assume that $\Gamma$ is a discrete group of automorphisms of $X$ which has no fixed points and for which $X/\Gamma$ is a compact quotient manifold. Then the celebrated Hirzebruch's proportionality principle asserts that the corresponding Chern numbers of $X/\Gamma$ and $\widetilde{X}$ are proportional with an explicitly determined proportionality factor (\cite[p. 137]{Hi1}, \cite{Hi0}, \cite[\S 22.3]{Hi2}).
\begin{theorem}[Hirzebruch's proportionality principle]\label{HPP1}
For each partition $\lambda$ of weight $n$, denote by $c_{\lambda}(X/\Gamma)$ and $c_{\lambda}(\widetilde{X})$ the respective Chern numbers of $X/\Gamma$ and $\widetilde{X}$ with respect to the partition $\lambda$. Then we have
\be c_{\lambda}(X/\Gamma)=\chi^0(X/\Gamma)\cdot c_{\lambda}(\widetilde{X}), \qquad\forall~\lambda,\nonumber\ee
where the proportionality factor is precisely the Todd genus $\chi^0(X/\Gamma)$ of $X/\Gamma$.
In particular,
\be\label{quotient}\chi_y(X/\Gamma)=\chi^0(X/\Gamma)\cdot\chi_y(\widetilde{X}).\ee
\end{theorem}

What we need in the proof of Theorem \ref{theorem1} is only a very special case of Theorem \ref{HPP1}, which we record in the following example.
\begin{example}\label{example}
Take the bounded homogeneous symmetric domain $X=\mathbb{B}^n$, the unit ball in $\mathbb{C}^n$. Then its compact dual is $\widetilde{X}=\mathbb{C}P^n$ and the proportionality factor $\chi^0(\mathbb{B}^n/\Gamma)=(-1)^n$ by our assumption in Theorem \ref{theorem1}. Therefore (\ref{quotient}) implies that
\be\chi_y(\mathbb{B}^n/\Gamma)=(-1)^n\cdot\chi_y(\mathbb{C}P^n),\nonumber\ee
\end{example}
and consequently by Proposition \ref{ch} we have
\be\label{example2}K_j(\mathbb{B}^n/\Gamma)
=(-1)^nK_j(\mathbb{C}P^n).\ee

\section{$L^2$-Hodge numbers and vanishing-type results}\label{section4}
In this section we briefly review the basic facts on $L^2$-Hodge numbers and indicate how to apply Atiyah's $L^2$-index theorem to obtain the relationship between $L^2$-Hodge numbers and the ordinary ones. The discussions here are sketchy and only for our later purpose. For a thorough treatment on these materials we refer the reader to the excellent book \cite{Lu}.

\subsection{$L^2$-Hodge numbers}
We assume throughout this subsection that $(M,g,J)$ is a compact complex $n$-dimensional manifold with a Hermitian metric $g$, and $$\pi:~(\widetilde{M},\widetilde{g},\widetilde{J})\longrightarrow
(M,g,J)$$
its universal covering with $\pi_1(M)$ as an isometric group of deck transformations.

Let $\mathcal{H}^{p,q}_{(2)}(\widetilde{M})$ be the spaces of $L^2$-harmonic $(p,q)$-forms on $L^2\Omega^{p,q}(\widetilde{M})$, the squared integrable $(p,q)$-forms on $(\widetilde{M},\widetilde{g})$, and denote by $$\text{dim}_{\pi_1(M)}\mathcal{H}^{p,q}_{(2)}(\widetilde{M})$$
the \emph{Von Neumann dimension} of $\mathcal{H}^{p,q}_{(2)}(\widetilde{M})$ with respect to $\pi_1(M)$, which is a \emph{nonnegative real number} in our situation. Its precise definition is not important in our article but only the following two basic facts are needed.
\begin{lemma}
\be\label{fact1}\text{dim}_{\pi_1(M)}\mathcal{H}^{p,q}_{(2)}(\widetilde{M})=0 \Longleftrightarrow\mathcal{H}^{p,q}_{(2)}(\widetilde{M})=\{0\},\ee
and $\text{dim}_{\pi_1(M)}(\cdot)$ is additive:
\be\label{fact2}\text{$\text{dim}_{\pi_1(M)}(A\oplus B)=\text{dim}_{\pi_1(M)}A+\text{dim}_{\pi_1(M)}B$}.\ee
\end{lemma}
Then the $L^2$-Hodge numbers of $M$, denoted by $h^{p,q}_{(2)}(M)$, are defined to be
$$h^{p,q}_{(2)}(M):=\text{dim}_{\pi_1(M)}
\mathcal{H}^{p,q}_{(2)}(\widetilde{M})
\in\mathbb{R}_{\geq0},\qquad(0\leq p,q\leq n).$$

It turns out that $h^{p,q}_{(2)}(M)$ are independent of the Hermitian metric $g$ and depend only on $(M,J)$.

The Dolbeault-type operators $D_p$ in (\ref{GDC}) can be lifted to $(\widetilde{M},\widetilde{g},\widetilde{J})$: $$\widetilde{D_p}:~\bigoplus_{\text{$q$ even}}L^2\Omega^{p,q}(\widetilde{M})\longrightarrow\bigoplus_{\text{$q$ odd}}L^2\Omega^{p,q}(\widetilde{M}),$$ and one can define the \emph{$L^2$-index} of the lifted operators $\widetilde{D_p}$ by
\be
\begin{split}
\text{ind}_{\pi_1(M)}(\widetilde{D_p}):=&
\text{dim}_{\pi_1(M)}(\text{ker}\widetilde{D_p})-
\text{dim}_{\pi_1(M)}(\text{coker}\widetilde{D_p})\\
=&\text{dim}_{\pi_1(M)}\big[\bigoplus_{\text{$q$ even}}\mathcal{H}^{p,q}_{(2)}(\widetilde{M})\big]
-\text{dim}_{\pi_1(M)}\big[\bigoplus_{\text{$q$ odd}}\mathcal{H}^{p,q}_{(2)}(\widetilde{M})\big]\\
=&\sum_{\text{$q$ even}}\text{dim}_{\pi_1(M)}\mathcal{H}^{p,q}_{(2)}(\widetilde{M})
-\sum_{\text{$q$ odd}}\text{dim}_{\pi_1(M)}\mathcal{H}^{p,q}_{(2)}(\widetilde{M})\qquad\big(\text{by $(\ref{fact2})$}\big)\\
=&\sum_{q=0}^n(-1)^qh^{p,q}_{(2)}(M).
\end{split}\nonumber\ee

The celebrated $L^2$-index theorem of Atiyah (\cite{At}) asserts that
$$\text{ind}(D_p)=\text{ind}_{\pi_1(M)}(\widetilde{D_p})
$$
and so we have the following crucial identities between $\chi^p(M)$ and the $L^2$-Hodge numbers $h^{p,q}_{(2)}(M)$:
\be\label{factchiphodgenumbers}
\chi^p(M)=\sum_{q=0}^n(-1)^qh^{p,q}_{(2)}(M).\ee

\subsection{Vanishing and non-vanishing type results}
The following result is the main theorem in Gromov's seminal article \cite[p. 283]{Gr}.
\begin{theorem}[Gromov]\label{gromovtheorem}
Let $M$ be a complex $n$-dimensional K\"{a}hler hyperbolic manifold. Then the spaces of $L^2$-harmonic $(p,q)$-forms on its universal covering $\widetilde{M}$ satisfy
\begin{eqnarray}
\left\{ \begin{array}{ll}
\mathcal{H}^{p,q}_{(2)}(\widetilde{M})=\{0\},\qquad p+q\neq n,\\
\mathcal{H}^{p,q}_{(2)}(\widetilde{M})\neq\{0\},\qquad p+q=n,
\end{array} \right.\nonumber
\end{eqnarray}
which, via the fact (\ref{fact1}), is equivalent to
\begin{eqnarray}\label{gromovvanishing}
\left\{ \begin{array}{ll}
h^{p,q}_{(2)}(M)=0,\qquad p+q\neq n,\\
h^{p,q}_{(2)}(M)>0,\qquad p+q=n.
\end{array} \right.
\end{eqnarray}
\end{theorem}
\begin{remark}
The proof for the vanishing type results in the first situations $p+q\neq n$ is a direct application of the $L^2$ version's Lefschetz theorem and is not difficult (\cite[p. 273, 1.2.B]{Gr}), where the existence of a $d$-bounded K\"{a}hler form on $\widetilde{M}$ plays a dominant role. The real hard part is the non-vanishing results in the second situations $p+q=n$, where a careful analysis on the lower bound of the eigenvalues of the Laplacian on $L^2$-harmonic forms was carried out in \cite[p. 274-285]{Gr}.
\end{remark}
A direct consequence of Theorem \ref{gromovtheorem} is the solution of the Hopf conjecture in the K\"{a}hlerian case (\cite[p. 267]{Gr}):
\be\label{hopfconjecture2}
\begin{split}
(-1)^n\chi(M)=&(-1)^n\sum_{p}(-1)^p\chi^p(M)\\
=&\sum_ph^{p,n-p}_{(2)}(M)>0.\qquad\big(\text{by (\ref{factchiphodgenumbers}) and (\ref{gromovvanishing})}\big)
\end{split}\ee

By extending the arguments in the proof of Theorem \ref{gromovtheorem} in the first situations $p+q\neq n$, Cao-Xavier and Jost-Zuo independently obtained the following (\cite{CX}, \cite{JZ})
\begin{theorem}[Cao-Xavier, Jost-Zuo]\label{CaoXavierJostZuotheorem}
Let $M$ be a complex $n$-dimensional K\"{a}hler non-elliptic manifold. Then $\mathcal{H}^{p,q}_{(2)}(\widetilde{M})=\{0\}$ when $p+q\neq n$, i.e.,
\be\label{cxjz}h^{p,q}_{(2)}(M)=0,\qquad p+q\neq n.\ee
\end{theorem}
This implies from (\ref{hopfconjecture2}) that $(-1)^n\chi(M)\geq0$ and thus settles the nonnegative version's Hopf conjecture in the K\"{a}hlerian case.

\section{Proofs of main results}\label{section5}
With the background materials prepared in Sections \ref{section3} and \ref{section4}, we are ready to prove our main results in this section.
\subsection{Proofs of Theorems \ref{theorem1} and \ref{theorem2}}
In this subsection we mainly show Theorem \ref{theorem1}, from whose process Theorem \ref{theorem2} follows easily.

Assume now that $M$ is a complex $n$-dimensional K\"{a}hler hyperbolic manifold. Then
\be\label{1}
\begin{split}
\chi^p(M)=&\sum_{q=0}^n(-1)^qh^{p,q}_{(2)}(M)\qquad\big(\text{by (\ref{factchiphodgenumbers})}\big)\\
=&(-1)^{n-p}h^{p,n-p}_{(2)}(M).\qquad\big(\text{by (\ref{gromovvanishing})}\big)
\end{split}\ee

Note that $\chi^p(M)$ is by definition an integer. On the other hand,  we know from (\ref{gromovvanishing}) that $h^{p,n-p}_{(2)}(M)$ is a
\emph{positive} real number. Therefore the equality (\ref{1}) implies that $h^{p,n-p}_{(2)}(M)$ is indeed a positive integer and thus
\be\label{2}h^{p,n-p}_{(2)}(M)\geq1,\qquad 0\leq p
\leq n.\ee

Still following the notation in (\ref{chiy-1}), we have
\be\label{3}
\begin{split}
(-1)^n\sum_{j=0}^nK_j(M)\cdot(y+1)^j=&
(-1)^n\chi_y(M)\\
=&(-1)^n\sum_{p=0}^n\chi^p(M)\cdot y^p\\
=&\sum_{p=0}^nh^{p,n-p}_{(2)}(M)\cdot(-y)^p.\qquad\big(\text{by (\ref{1})}\big)
\end{split}
\ee

Now comparing the coefficients of the Taylor expansion at $y=-1$ on both sides of (\ref{3}) yields
\be\label{4}
\begin{split}
(-1)^nK_j(M)=&\frac{\Big[\sum_{p=0}^nh^{p,n-p}_{(2)}(M)(-y)^p\Big]^{(j)}}
{j!}\Big|_{y=-1}\qquad(0!:=1)\\
=&(-1)^j\sum_{p=j}^n{p\choose j}h^{p,n-p}_{(2)}(M).
\end{split}
\ee

This implies that
\be\label{5}
\begin{split}
(-1)^{n+j}K_{j}(M)=&\sum_{p=j}^n{p\choose j}h^{p,n-p}_{(2)}(M)\\
\geq&\sum_{p=j}^n{p\choose j}\qquad\big(\text{by (\ref{2})}\big)\\
=&(-1)^j\frac{\Big[\sum_{p=0}^n(-y)^p\Big]^{(j)}}
{j!}\Big|_{y=-1}\\
=&(-1)^j\frac{\Big[\chi_y(\mathbb{C}P^n)\Big]^{(j)}}
{j!}\Big|_{y=-1}\\
=&(-1)^jK_j(\mathbb{C}P^n).
\end{split}
\ee

Now we define
$$A_i(c_1,\ldots,c_n)[M]:=(-1)^nK_{2i}(M),\qquad 0\leq i\leq[\frac{n}{2}].$$
Then it follows from (\ref{5}) that
\be
\begin{split}
A_i(c_1,\ldots,c_n)[M]\geq&(-1)^nA_{i}
(c_1,\ldots,c_n)[\mathbb{C}P^n]\\
=&(-1)^nA_i\Big({n+1\choose 1},\ldots,{n+1\choose n}\Big),\qquad 0\leq i\leq[\frac{n}{2}],\end{split}\nonumber\ee
which produce the desired Chern number inequalities (\ref{chernnumberinequality}) and, together with Proposition \ref{ch}, the formulas for the first three terms in Theorem \ref{theorem1}.

Clearly the equality case in (\ref{5}) holds if and only if
$$h^{p,n-p}_{(2)}(M)=1,\qquad j\leq p\leq n,$$
which, via (\ref{1}), is equivalent to
\be\label{6}\chi^p(M)=(-1)^{n-p},\qquad j\leq p\leq n,\ee
which precisely give the equality characterization (\ref{equalitycase1}) in Theorem \ref{theorem1}. Also note that, if $j\leq[\frac{n+1}{2}]$, the relations $\chi^p=(-1)^n\chi^{n-p}$ in (\ref{chiprelation}) tell us that the $n-j+1$ equalities in (\ref{6}) indeed are equivalent to $\chi^p(M)=(-1)^{n-p}$ for \emph{all} $p$, i.e.,
$$\chi_y(M)=(-1)^n\sum_{p=0}^n(-y)^n=(-1)^n\chi_y(\mathbb{C}P^n).$$
This gives the desired equality characterizations in (\ref{equalitycase2}) as $2i\leq[\frac{n+1}{2}]$ is equivalent to $i\leq[\frac{n+1}{4}]$.

In order to complete the proof of Theorem \ref{theorem1}, it suffices to show that the equality cases in (\ref{chernnumberinequality}) can be realized by some compact quotients of the unit ball in $\mathbb{C}^n$. But it has been done via (\ref{example2}) by applying the Hirzebruch's proportionality principle in Example \ref{example}. This completes the proof of Theorem \ref{theorem1}.

The proof above can be completely carried over to show Theorem \ref{theorem2} for K\"{a}hler non-elliptic manifolds by applying the vanishing-type results (\ref{cxjz}) in Theorem \ref{CaoXavierJostZuotheorem}. The only difference is that in this case the conditions in (\ref{2}) are unavailable and so accordingly the inequality (\ref{5}) has to be weakened to
$$(-1)^{n+j}K_j(M)\geq0,$$
which lead to the desired (\ref{nonelliptic}).

\subsection{Proof of Theorem \ref{theorem3}}
Let us complete this article by proving Theorem \ref{theorem3} in this last subsection.

It is well-known, by combining the Kodaira vanishing theorem and the Hirzebruch-Riemann-Roch theorem, that a projective manifold with ample canonical bundle is of general type. Conversely, the canonical bundle of a projective manifold of general type may not be ample. The following fact says that it is the case if an extra condition is assumed.
\begin{lemma}\label{nonstandardlemma}
If a projective manifold of general type contains no rational curves, then its canonical bundle is ample.
\end{lemma}
\begin{proof}
This fact should be well-known to experts. For example, this was listed in \cite[p. 219]{De} as an exercise with hints and the details were carried out in the proof in \cite[Thm 2.11]{CY}.
\end{proof}

With this lemma in hand, we now proceed to prove Theorem \ref{theorem3}.
\begin{proof}
First note that the manifold $M$ in question is projective. Indeed, $M$ being of general type implies that its canonical bundle is big and so $M$ is Moishezon (\cite[p. 88]{MM}). Together with the K\"{a}hlerness condition we conclude from Moishezon's theorem that $M$ is projective (cf. \cite[p. 95]{MM}).

In view of Lemma \ref{nonstandardlemma}, it now suffices to show that $M$ contains no rational curves. The following arguments are parallel to those in \cite[Thm 2.11]{CY}.

Since $M$ is K\"{a}hler exact, there exists a K\"{a}hler form $\omega$ on it such that $\pi^{\ast}\omega=d\beta$ for some $1$-form $\beta$ on $\widetilde{M}$. Assume that $f:~\mathbb{C}P^1\longrightarrow M$ is a holomorphic map and we want to show that $f$ is a constant map, i.e., $f^{\ast}(\omega)\equiv0$. Let $\pi:~\widetilde{M}\longrightarrow M$ be the universal covering. Due to the simple-connectedness of $\mathbb{C}P^1$ the map $f$ admits a lifting $\widetilde{f}$ to $\widetilde{M}$,
$$\xymatrix{
                &         \widetilde{M} \ar[d]^{\pi}     \\
\mathbb{C}P^1 \ar[ur]^{\widetilde{f}} \ar[r]_{f} & M, }$$
i.e., $f=\pi\circ\widetilde{f}$. Therefore
\be\int_{\mathbb{C}P^1}f^{\ast}\omega=
\int_{\mathbb{C}P^1}(\pi\circ\widetilde{f})^{\ast}(\omega)=
\int_{\mathbb{C}P^1}\widetilde{f}^{\ast}(\pi^{\ast}\omega)
=\int_{\mathbb{C}P^1}\widetilde{f}^{\ast}(d\beta)
=\int_{\mathbb{C}P^1}d(\widetilde{f}^{\ast}\beta)=0.\nonumber\ee
This means that $f^{\ast}(\omega)\equiv0$ and so $f$ is a constant map, which completes the proof of Theorem \ref{theorem3} and this article.
\end{proof}

\end{document}